\newif\ifJOURNAL
\JOURNALfalse
\newif\ifarXiv
\arXivfalse
\newif\ifWP
\WPfalse
\newif\ifFULL
\FULLfalse
\newif\ifLATIN
\LATINfalse

\arXivtrue

\ifarXiv\LATINtrue\fi	

\newif\ifnotJOURNAL	
\notJOURNALtrue
\ifJOURNAL\notJOURNALfalse\fi

\newif\ifnotarXiv	
\notarXivtrue
\ifarXiv\notarXivfalse\fi

\newif\ifTR		
\TRfalse
\ifarXiv\TRtrue\fi
\ifWP\TRtrue\fi
\newif\ifnotTR
\notTRtrue
\ifarXiv\notTRfalse\fi
\ifWP\notTRfalse\fi

\newif\ifnotFULL	
\notFULLtrue
\ifFULL\notFULLfalse\fi

\newif\ifnotLATIN	
\notLATINtrue
\ifLATIN\notLATINfalse\fi

\ifnotLATIN
  \newcommand{\Kabanov}{kabanov/etal:1977}
  \newcommand{\VovkDAN}{vovk:1987criterion-full}
  \newcommand{\VovkTVP}{vovk:1987law-full}
  \newcommand{\VovkPPI}{vovk:1989full}
  \newcommand{\AmariNagaoka}{amari/nagaoka:2000}
\fi
\ifLATIN
  \newcommand{\Kabanov}{kabanov/etal:1977latin}
  \newcommand{\VovkDAN}{vovk:1987criterion}
  \newcommand{\VovkTVP}{vovk:1987law}
  \newcommand{\VovkPPI}{vovk:1989}
  \newcommand{\AmariNagaoka}{amari/nagaoka:2000latin}
\fi

\ifJOURNAL
\documentclass[smallextended,natbib,runningheads]{svjour3}
\journalname{Annals of the Institute of Statistical Mathematics}
\smartqed  
\usepackage{graphicx}
\usepackage{amsmath,amsfonts,amssymb,latexsym}
\newcommand{\Extra}[1]{}
\fi

\ifarXiv
\documentclass{article}
\usepackage{amsmath,amsfonts,amssymb,latexsym,graphicx,natbib}
\newcommand{\Extra}[1]{}
\fi

\ifWP
\documentclass{gtarticle}
\usepackage{amsmath,amsfonts,amssymb,latexsym,epsfig,graphicx,natbib}
\renewcommand{\Extra}[1]{#1}
\fi

\ifFULL
\usepackage{color}
\renewcommand{\Extra}[1]{\blue{#1}}

\newcommand{\blue}[1]{\textcolor{blue}{#1}}
\newcommand{\bluebegin}{\begingroup\color{blue}}
\newcommand{\blueend}{\endgroup}

\fi

\allowdisplaybreaks[3]

\ifnotLATIN
  \input{OT2enc.def}
  
  \usepackage{CJK}
\fi

\ifJOURNAL
  \renewcommand{\I}{^{\mathrm{I}}}
\fi
\ifarXiv
  \newcommand{\I}{^{\mathrm{I}}}
\fi
\newcommand{\II}{^{\mathrm{II}}}

\newcommand{\st}{\mathrel{\!|\!}}
\newcommand{\givn}{\mathrel{|}}
\newcommand{\contig}{\mathrel{\lhd}}
\newcommand{\separ}{\mathrel{\bigtriangleup}}
\newcommand{\para}{\mathbin{\parallel}}

\newcommand{\dd}{\mathrm{d}}		

\newcommand{\K}{\mathcal{K}}		
\newcommand{\PPP}{\mathcal{P}}		

\newcommand{\FFF}{\mathcal{F}}

\newcommand{\SSS}{\mathcal{S}}

\DeclareMathOperator{\III}{\mathbb{I}}		

\newcommand{\bbbp}{\mathbb{P}}		
\DeclareMathOperator{\Prob}{\bbbp}

\newcommand{\bbbr}{\mathbb{R}}		
\ifnotJOURNAL
  \newtheorem{lemma}{Lemma}
  \newtheorem{proposition}{Proposition}
  \newtheorem{corollary}{Corollary}
  \newtheorem{theorem}{Theorem}
  \newenvironment{proof}
    {\trivlist\item[\hskip\labelsep\textbf{Proof}]}
    {\endtrivlist}
\newenvironment{Proof}[1]
  {\trivlist\item[\hskip\labelsep\textbf{Proof #1:}]}
  {\endtrivlist}
\fi
\newcommand{\boxforqed}{\rule{.3em}{1.5ex}}
\newcommand{\qedtext}{\unskip\nobreak\hfil
  \penalty50\hskip1em\null\nobreak\hfil\boxforqed
  \parfillskip=0pt\finalhyphendemerits=0\endgraf}

\newenvironment{remark*}
  {\trivlist\item[\hskip\labelsep{\bfseries Remark}]\relax}
  {\endtrivlist}

\newlength{\IndentI}
\newlength{\IndentII}
\newlength{\IndentIII}
\newlength{\IndentIV}
\newlength{\WidthI}
\newlength{\WidthII}
\newlength{\WidthIII}
\newlength{\WidthIV}
\ifarXiv
\title{Merging of opinions in game-theoretic probability}
\author{Vladimir Vovk\\
\texttt{vovk{\rm@}cs.rhul.ac.uk}\\
\texttt{http://vovk.net}}
\fi

\ifWP
\title{Merging of opinions in game-theoretic probability}
\author{Vladimir Vovk}

\fi

\begin{document}

\maketitle

\begin{abstract}
  This paper gives game-theoretic versions of several results on ``merging of opinions''
  obtained in measure-theoretic probability and algorithmic randomness theory.
  An advantage of the game-theoretic versions over the measure-theoretic results
  is that they are pointwise,
  their advantage over the algorithmic randomness results
  is that they are non-asymptotic,
  but the most important advantage over both is that they are very constructive,
  giving explicit and efficient strategies for players
  in a game of prediction.
\end{abstract}

\section{Introduction}
\label{sec:introduction}

\ifTR
  The idea that the predictions made by two forecasters
  will become closer with increasing information
  goes back at least to de Finetti
  (\citeyear{definetti:1937}, Chapter V);
  see also Savage \citeyear{savage:1954}, \S3.6.
  De Finetti's assumption was that both forecasters
  compute their predictions from exchangeable probability measures
  that are not too close to a power probability measure;
  in detail he considered only binary prediction from this point of view.
  Later it became quite popular in Bayesian statistics
  since it renders an element of objectivity
  to the subjective probability measures.

  A similar phenomenon is also known in economics
  under the name of ``Hotelling's law'' (after \citealt{hotelling:1929})
  or the ``principle of minimum differentiation''.
\fi

The first general mathematical result about convergence of predictions
made by successful forecasters appears to be
Blackwell and Dubins's paper (\citeyear{blackwell/dubins:1962}).
Blackwell and Dubins's result was about infinite-horizon forecasting,
whereas in this paper we will be interested in one-step ahead forecasting.
\ifTR
  (In game-theoretic probability,
  Blackwell and Dubins's setting is much less natural,
  since there are no stochastic assumptions imposed
  on what happens outside the protocol.)
\fi
An important paper in this direction was \citet{\Kabanov}
(see also \citealt{shiryaev:1996}, Sect.~VII.6,
\citealt{jacod/shiryaev:2003-local}, and \citealt{greenwood/shiryaev:1985}),
which generalized earlier results by \citet{kakutani:1948}
and by \citet{hajek:1958} and \citet{feldman:1958}.
The main disadvantage of the approach of these papers
is that it is ``bulk'',
stated in terms of absolute continuity and singularity
of probability measures.

One way to obtain pointwise results about merging of opinions
is to use the algorithmic theory of randomness.
The first result of this kind was proved by \citet{dawid:1985}
(who refers to it as ``Jeffreys's law'' in \citealt{dawid:1984} and \citealt{dawid:2004}).
Dawid's result was for his version of von Mises's notion of randomness
based on subsequence selection rules,
and so his notion of merging was rather weak.
A result based on the standard notion of randomness
was obtained by \citet{\VovkDAN}
and later extended by \citet{fujiwara:2007}.

The algorithmic randomness approach has two major weaknesses.
First, since it is based on the notion of computability,
it imposes heavy restrictions on the types of measurable spaces
it can deal with
(typically one considers just finite or countable observation spaces $\Omega$).
Second, it is asymptotic in the sense that it never provides us
with explicit inequalities.
The von Mises-type notion of randomness
\Extra{is hopelessly stuck at infinity
(to use Shafer's \citep{vovk:1993logic} expression):
it}
does not assert anything at all about finite sequences of observations.
But even the modern definitions
(such as those due to Martin-L\"of, Levin, and Schnorr)
are based on a notion
(the deficiency of randomness)
that is defined only to within an additive constant
and so can only be applied to finite sequences \emph{en masse}.

The game-theoretic approach to probability was suggested in \citet{vovk:1993logic}
and developed in, e.g.,
\citet{dawid/vovk:1999},
\citet{shafer/vovk:2001},
\citet{kumon/etal:2007}.
This approach makes it possible to use all flexibility
of the algorithmic randomness approach
without paying its high price;
in particular, all our statements are either non-asymptotic
or can be stated in a non-asymptotic manner.

\section{Merging of opinions as criterion of success}
\label{sec:criterion}

Let $\Omega$ be a measurable space
and $\PPP(\Omega)$ stand for the set of all probability measures on $\Omega$;
elements of $\Omega$ will be called \emph{observations}
and measurable subsets of $\Omega$ will be called \emph{local events}.
Suppose we have two forecasters at each step issuing probability forecasts
for the next observation $\omega_n\in\Omega$ to be chosen by reality.
The game-theoretic process of testing the forecasters' predictions can be represented
in the following form.

\medskip

\noindent
\textsc{Competitive testing protocol}

\noindent
\textbf{Players:} Reality, Forecaster I, Sceptic I, Forecaster II, Sceptic II

\noindent
\textbf{Protocol:}

\parshape=11
\IndentI   \WidthI
\IndentI   \WidthI
\IndentI   \WidthI
\IndentII  \WidthII
\IndentII  \WidthII
\IndentII  \WidthII
\IndentII  \WidthII
\IndentII  \WidthII
\IndentII  \WidthII
\IndentII  \WidthII
\IndentI   \WidthI
\noindent
$\K\I_0 := 1$.\\
$\K\II_0 := 1$.\\
FOR $n=1,2,\dots$:\\
  Forecaster I announces $P\I_n\in\PPP(\Omega)$.\\
  Forecaster II announces $P\II_n\in\PPP(\Omega)$.\\
  Sceptic II announces $f\II_n:\Omega\to[0,\infty]$
	such that $\int f\II_n \dd P\II_n=1$.\\
  Sceptic I announces $f\I_n:\Omega\to[0,\infty]$
	such that $\int f\I_n \dd P\I_n=1$.\\
  Reality announces $\omega_n\in\Omega$.\\
  $\K\I_n := \K\I_{n-1} f\I_n (\omega_n)$.\\
  $\K\II_n := \K\II_{n-1} f\II_n (\omega_n)$.\\
END FOR

\medskip

\noindent
The predictions output by Forecaster I are tested by Sceptic I,
and the predictions output by Forecaster II are tested by Sceptic II.
The Sceptics' success in detecting inadequacy of the Forecasters' predictions
is measured by their \emph{capital}, $\K^{\textrm{I}}_n$ and $\K^{\textrm{II}}_n$,
respectively.
The initial capital is $1$ and the game is fair
from the point of view of the Forecasters.
The value of $\K\I_n$
(resp.\ $\K\II_n$) is interpreted as the degree
to which Sceptic I (resp.\ Sceptic II)
managed to discredit Forecaster I's (resp.\ Forecaster II's) predictions.
The requirement that the Sceptics
choose functions $f\I_n$ and $f\II_n$ taking nonnegative values
reflects the restriction that they should never risk bankruptcy
by gambling more than their current capital.

In our protocol we allow infinite values for the functions chosen by the Sceptics
and, therefore, infinite values for their capital.
We will use the convention $0\infty:=0$.

Let $\alpha\notin\{-1,1\}$.
For two probability measures $P\I$ and $P\II$ on $\Omega$
we define the \emph{$\alpha$-divergence} between them as
\begin{equation}\label{eq:divergence}
  D^{(\alpha)}
  \left(
    P\I\para P\II
  \right)
  :=
  \frac{4}{1-\alpha^2}
  \left(
    1
    -
    \int_{\Omega}
      (\beta\I(\omega))^{\frac{1-\alpha}{2}}
      (\beta\II(\omega))^{\frac{1+\alpha}{2}}
    Q(\dd\omega)
  \right)
\end{equation}
(with the same convention $0\infty:=0$)
where $Q$ is any probability measure on $\Omega$
such that $P\I\ll Q$ and $P\II\ll Q$,
$\beta\I$ is any version of the density of $P\I$ w.r.\ to $Q$
and $\beta\II$ is any version of the density of $P\II$ w.r.\ to $Q$.
(For example, one can set $Q:=(P\I+P\II)/2$;
it is clear that the value of the integral
does not depend on the choice of $Q$, $\beta\I$ and $\beta\II$.)
The expression (\ref{eq:divergence}) is always nonnegative:
see, e.g., \citet{\AmariNagaoka}.
An important special case is the \emph{Hellinger distance},
corresponding to $\alpha=0$ and also given by the formula
\begin{equation*}
  D^{(0)}
  \left(
    P\I\para P\II
  \right)
  =
  2
  \int_{\Omega}
    \left(
      \sqrt{\beta\I(\omega)}
      -
      \sqrt{\beta\II(\omega)}
    \right)^2
  Q(\dd\omega).
\end{equation*}
(Sometimes ``Hellinger distance'' refers to $\frac12D^{0}(P\I,P\II)$,
as in \citealt{\VovkDAN},
or to $\sqrt{D^{0}(P\I,P\II)}$.)

For simplicity,
in the main part of this section
we will only consider the case where Forecaster II is ``timid''
on the given play of the game,
in the sense that he does not deviate too much from Forecaster I.
Formally,
Forecaster II is \emph{timid}
if, for all $n$, $P\I_n\ll P\II_n$
(intuitively,
if Forecaster II never declares a local event null
unless it is already null according to Forecaster I).
It should be remembered that the assumption of timidity
is always imposed on the realized play of the game
rather than on Forecaster II's strategy
(in general, Forecaster II is not assumed to follow a strategy).

The following asymptotic result will be proved in Sect.~\ref{sec:criterion-proof}
(its counterpart in the algorithmic theory of randomness
has been recently proved by \citealt{fujiwara:2007}, Theorem 3;
in the special case $\alpha=0$
it was obtained in \citealt{\VovkDAN}).
We will say that Sceptic I (resp.\ Sceptic II)
\emph{becomes infinitely rich} if $\lim_{n\to\infty}\K\I_n=\infty$
(resp.\ $\lim_{n\to\infty}\K\II_n=\infty$).
\renewcommand{\thetheorem}{1a}
\begin{theorem}\label{thm:criterion}
  Let $\alpha\in(-1,1)$.
  In the competitive testing protocol:
  \begin{enumerate}
  \item\label{it:c1}
    The Sceptics have a joint strategy guaranteeing
    that at least one of them will become infinitely rich
    if
    \begin{equation}\label{eq:far}
      \sum_{n=1}^{\infty}
      D^{(\alpha)}
      \left(
        P\I_n\para P\II_n
      \right)
      =
      \infty
    \end{equation}
    and Forecaster II is timid.
  \item\label{it:c2}
    Sceptic I has a strategy guaranteeing
    that he will become infinitely rich
    if
    \begin{equation}\label{eq:close}
      \sum_{n=1}^{\infty}
      D^{(\alpha)}
      \left(
        P\I_n \para P\II_n
      \right)
      <
      \infty,
    \end{equation}
    Sceptic II becomes infinitely rich,
    and Forecaster II is timid.
  \end{enumerate}
\end{theorem}
\renewcommand{\thetheorem}{\arabic{theorem}}
Sceptic I (resp.\ Sceptic II) becoming infinitely rich in this theorem
can also be understood as $\limsup_{n\to\infty}\K\I_n=\infty$
(resp.\ $\limsup_{n\to\infty}\K\II_n=\infty$);
in the next subsection we will see
that this understanding of ``infinitely rich'' leads to an equivalent statement.

\begin{remark*}
  Fujiwara (\citeyear{fujiwara:2007}, Section 3.1) gives a simple example
  showing that Theorem \ref{thm:criterion} (namely, its Part \ref{it:c1})
  cannot be extended to the case $\left|\alpha\right|\ge1$.
\end{remark*}

Before discussing the intuition behind Theorem \ref{thm:criterion},
it will be convenient to introduce some terminology (in part informal).
We will say that Forecaster I (resp.\ Forecaster II)
is \emph{successful}
(for a particular play of the game)
if Sceptic I (resp.\ Sceptic II)
does not become infinitely rich.
We say that a Forecaster is \emph{reliable}
if we believe, even before the start of the game,
that he will be successful.
For example,
the Forecaster might know the true stochastic mechanism producing the observations,
or his predictions may be computed from a well-tested theory.

Part \ref{it:c1} of the theorem says that either (\ref{eq:close}) holds
or at least one of the Sceptics becomes infinitely rich.
Therefore, if both Forecasters are reliable,
we expect their predictions to be close in the sense of (\ref{eq:close}).

Suppose we only know that Forecaster I is reliable;
for concreteness,
let us impose on Reality the requirement that $\K\I_n$ should stay bounded.
If the Sceptics invest a fraction (arbitrarily small) of their initial capital
in strategies whose existence is guaranteed in Theorem \ref{thm:criterion},
we will have
\begin{equation*}
  \left(
    \limsup_{n\to\infty}
    \K\II_n
    <
    \infty
  \right)
  \Longleftrightarrow
  \left(
    \sum_{n=1}^{\infty}
    D^{(\alpha)}
    \left(
      P\I_n\para P\II_n
    \right)
    <
    \infty
  \right).
\end{equation*}
(In the context of algorithmic randomness theory
this equivalence is called a criterion of randomness
in \citealt{\VovkDAN} and \citealt{fujiwara:2007}.)

In the following sections we will see several elaborations
of Theorem \ref{thm:criterion}.
One useful interpretation of our results
is where Forecaster I computes his predictions
using some well-tested stochastic theory,
and we believe him to be reliable.
Forecaster II represents an alternative way of forecasting.
We will be interested in the relation
between the deviation of Forecaster II's predictions from Forecaster I's predictions
and the degree of the former's success,
as measured by Sceptic II's capital.

\subsection{Equivalence of the two senses of becoming infinitely rich}

Let us first simplify the competitive testing protocol:

\medskip

\noindent
\textsc{Testing protocol}

\noindent
\textbf{Players:} Reality, Forecaster, Sceptic

\noindent
\textbf{Protocol:}

\parshape=7
\IndentI   \WidthI
\IndentI   \WidthI
\IndentII  \WidthII
\IndentII  \WidthII
\IndentII  \WidthII
\IndentII  \WidthII
\IndentI   \WidthI
\noindent
$\K_0 := 1$.\\
FOR $n=1,2,\dots$:\\
  Forecaster announces $P_n\in\PPP(\Omega)$.\\
  Sceptic announces $f_n:\Omega\to[0,\infty]$
	such that $\int f_n \dd P_n=1$.\\
  Reality announces $\omega_n\in\Omega$.\\
  $\K_n := \K_{n-1} f_n (\omega_n)$.\\
END FOR

\medskip

\noindent
Now we have only one Forecaster and one Sceptic.
We again refer to $\K_n$ as Sceptic's \emph{capital} at time $n$.

The proof of the following lemma will give an efficient procedure
transforming a strategy for Sceptic into another strategy for Sceptic
such that the second strategy makes Sceptic infinitely rich in the sense $\lim\K_n=\infty$
whenever the first strategy makes him infinitely rich
in the sense $\limsup\K_n=\infty$.
If $\SSS$ is a strategy for Sceptic,
$P_1\omega_1P_2\omega_2\ldots$ is a sequence of moves by Forecaster and Reality,
and $n\in\{1,2,\ldots\}$,
let $\K_n(\SSS,P_1\omega_1P_2\omega_2\ldots)$ be Sceptic's capital
achieved when playing $\SSS$ against Forecaster and Reality
playing $P_1\omega_1P_2\omega_2\ldots$\,.
\begin{lemma}\label{lem:equivalence}
  For any strategy $\SSS$ for Sceptic
  there exists another strategy $\SSS'$ for Sceptic
  such that,
  for all $P_1\omega_1P_2\omega_2\ldots$,
  \begin{equation}\label{eq:lem-equivalence}
    \limsup_{n\to\infty}
    \K_n(\SSS,P_1\omega_1P_2\omega_2\ldots)
    =
    \infty
    \Longrightarrow
    \lim_{n\to\infty}
    \K_n(\SSS',P_1\omega_1P_2\omega_2\ldots)
    =
    \infty.
  \end{equation}
\end{lemma}
\begin{proof}
  This proof will use the argument from \citet{vovk/shafer:2005RSS}
  (the end of the proof of Theorem 3),
  which we learned from Sasha Shen;
  for another argument,
  see \citet{shafer/vovk:2001}, Lemma 3.1.

  Let $\SSS$ be any strategy for Sceptic.
  The transformed strategy $\SSS'$ works as follows.
  Start playing $\SSS$ until $\K_n$ exceeds $2$
  (play $\SSS$ forever if $\K_n$ never exceeds $2$).
  As soon as this happens,
  set $1$ aside and continue playing $\SSS$ with the initial active capital $\K_n-1$
  until the active capital exceeds $2$.
  As soon as this happens,
  set another $1$ aside (decreasing the active capital by this amount)
  and continue playing $\SSS$ until the active capital exceeds $2$,
  etc.

  Formally, if at the beginning of some step $n$
  the capital $\K_{n-1}$ attained by $\SSS'$ includes active capital $\K^{\textrm{act}}_{n-1}$,
  with $\K_{n-1}-\K^{\textrm{act}}_{n-1}$ set aside earlier,
  and $\SSS$ recommends move $f_n$,
  the move recommended by $\SSS'$ is
  \begin{equation*}
    f'_n
    :=
    \frac{\K^{\textrm{act}}_{n-1}}{\K_{n-1}}
    f_n
    +
    \frac{\K_{n-1}-\K^{\textrm{act}}_{n-1}}{\K_{n-1}},
  \end{equation*}
  so that Sceptic's capital becomes
  $\K^{\textrm{act}}_{n-1} f_n(\omega_n) + (\K_{n-1}-\K^{\textrm{act}}_{n-1})$
  at the end of step $n$.
  Now (\ref{eq:lem-equivalence}) follows from the fact that Sceptic
  will set aside another $1$ infinitely often
  when the antecedent of (\ref{eq:lem-equivalence}) is satisfied.
  \ifnotJOURNAL
    \qedtext
  \fi
\end{proof}

\subsection{General theorem about merging of opinions}

In this subsection we drop the assumption
that Forecaster II is timid;
unfortunately,
this will make the statement of the theorem more complicated
(it is easy to see that Theorem \ref{thm:criterion} itself
becomes false if the assumption of Sceptic II's timidity is removed).
Even the competitive testing protocol has to be modified.
Let us say that $(E\I,E\II)$,
where $E\I$ and $E\II$ are local events,
is an \emph{exceptional pair} for $P\I,P\II\in\PPP(\Omega)$
if $P\I(E\I)=0$, $P\II(E\II)=0$, and
\begin{equation*}
  P\I(E)=0
  \Longleftrightarrow
  P\II(E)=0
\end{equation*}
for all $E\subseteq\Omega\setminus(E\I\cup E\II)$.

\medskip

\noindent
\textsc{Modified competitive testing protocol}

\noindent
\textbf{Players:} Reality, Forecaster I, Sceptic I, Forecaster II, Sceptic II

\noindent
\textbf{Protocol:}

\parshape=12
\IndentI   \WidthI
\IndentI   \WidthI
\IndentI   \WidthI
\IndentII  \WidthII
\IndentII  \WidthII
\IndentII  \WidthII
\IndentII  \WidthII
\IndentII  \WidthII
\IndentII  \WidthII
\IndentII  \WidthII
\IndentII  \WidthII
\IndentI   \WidthI
\noindent
$\K\I_0 := 1$.\\
$\K\II_0 := 1$.\\
FOR $n=1,2,\dots$:\\
  Forecaster I announces $P\I_n\in\PPP(\Omega)$.\\
  Forecaster II announces $P\II_n\in\PPP(\Omega)$.\\
  Reality announces an exceptional pair $(E\I_n,E\II_n)$ for $P\I_n,P\II_n$.\\
  Sceptic II announces $f\II_n:\Omega\to[0,\infty]$
	such that $\int f\II_n \dd P\II_n=1$.\\
  Sceptic I announces $f\I_n:\Omega\to[0,\infty]$
	such that $\int f\I_n \dd P\I_n=1$.\\
  Reality announces $\omega_n\in\Omega$.\\
  $\K\I_n := \K\I_{n-1} f\I_n (\omega_n)$.\\
  $\K\II_n := \K\II_{n-1} f\II_n (\omega_n)$.\\
END FOR

\medskip

\noindent
The identity of the player who announces an exceptional pair
does not matter
as long as it is not one of the Sceptics.
One way to chose $(E\I,E\II)$ is to choose $\beta\I$ and $\beta\II$ first
and then set $E\I:=\{\beta\I=0\}$ and $E\II:=\{\beta\II=0\}$.

Without the condition of timidity of Forecaster II,
the condition of agreement (\ref{eq:close}) between the Forecasters
has to be replaced by
\begin{equation}\label{eq:general-close}
  \sum_{n=1}^{\infty}
  D^{(\alpha)}
  \left(
    P\I_n \para P\II_n
  \right)
  <
  \infty
  \quad\text{and}\quad
  \forall n:
  \omega_n \notin E\I_n\cup E\II_n.
\end{equation}

\renewcommand{\thetheorem}{1b}
\begin{theorem}\label{thm:general-criterion}
  Let $\alpha\in(-1,1)$.
  In the modified competitive testing protocol:
  \begin{enumerate}
  \item
    The Sceptics have a joint strategy guaranteeing
    that at least one of them will become infinitely rich
    if (\ref{eq:general-close}) is violated.
  \item
    Sceptic I has a strategy guaranteeing
    that he will become infinitely rich
    if (\ref{eq:general-close}) holds
    and Sceptic II becomes infinitely rich.
  \end{enumerate}
\end{theorem}
\renewcommand{\thetheorem}{\arabic{theorem}}
\addtocounter{theorem}{-1}

\section{Non-asymptotic version}
\label{sec:non-asymptotic}

In many cases it will be more convenient to use the following modification
of the $\alpha$-divergence (\ref{eq:divergence})
between two probability measures $P\I$ and $P\II$ on $\Omega$:
\begin{equation}\label{eq:log-divergence}
  D^{[\alpha]}
  \left(
    P\I\para P\II
  \right)
  :=
  \frac{4}{\alpha^2-1}
  \ln
  \int_{\Omega}
    (\beta\I(\omega))^{\frac{1-\alpha}{2}}
    (\beta\II(\omega))^{\frac{1+\alpha}{2}}
  Q(\dd\omega),
\end{equation}
where the constant $\alpha$ is different from $-1$ and $1$;
$
  D^{[\alpha]}
  \left(
    P\I\para P\II
  \right)
$
will also be referred to as $\alpha$-divergence.
The expression (\ref{eq:log-divergence}) is nonnegative:
this follows from the fact that (\ref{eq:divergence}) is nonnegative.
When $P\I$ and $P\II$ are close to each other
(in the sense that
\begin{equation}\label{eq:Hellinger-integral-1}
  \int_{\Omega}
    (\beta\I(\omega))^{\frac{1-\alpha}{2}}
    (\beta\II(\omega))^{\frac{1+\alpha}{2}}
  Q(\dd\omega),
\end{equation}
called the \emph{Hellinger integral of order $\frac{1-\alpha}{2}$}%
,
is close to $1$),
the ratio of (\ref{eq:log-divergence}) to (\ref{eq:divergence})
is close to $1$.
In any case,
the inequality $\ln x\le x-1$ (for $x\ge0$) implies that
\begin{equation}\label{eq:inequalities}
  \begin{aligned}
    \left|\alpha\right|
    <
    1
    &\Longrightarrow
    D^{(\alpha)}
    \left(
      P\I\para P\II
    \right)
    \le
    D^{[\alpha]}
    \left(
      P\I\para P\II
    \right)\\
    \left|\alpha\right|
    >
    1
    &\Longrightarrow
    D^{(\alpha)}
    \left(
      P\I\para P\II
    \right)
    \ge
    D^{[\alpha]}
    \left(
      P\I\para P\II
    \right).
  \end{aligned}
\end{equation}
In principle, it is possible that
$
  D^{(\alpha)}
  \left(
    P\I\para P\II
  \right)
  =
  \infty
$:
this happens when the Hellinger integral in (\ref{eq:log-divergence}) is zero
(for $\left|\alpha\right|<1$)
or infinity
(for $\left|\alpha\right|>1$).
\begin{remark*}\label{p:renyi}
The version (\ref{eq:log-divergence}) coincides,
to within a constant factor and reparameterization,
with R\'enyi's (\citeyear{renyi:1961}) information gain,
which in our context can be written as
\begin{equation}\label{eq:renyi-divergence}
  D_{\alpha}
  \left(
    P\I,
    P\II
  \right)
  :=
  \frac{1}{\alpha-1}
  \log
  \int_{\Omega}
    (\beta\I(\omega))^{\alpha}
    (\beta\II(\omega))^{1-\alpha}
  Q(\dd\omega),
  \quad
  \alpha>0,
  \enspace
  \alpha\ne1,
\end{equation}
$\log$ standing for the binary logarithm.
However, we will never use the definition (\ref{eq:renyi-divergence})
in this paper;
an important advantage of (\ref{eq:log-divergence})
is that, for any constants $\alpha_1$ and $\alpha_2$,
the ratio of the divergences $D^{[\alpha_1]}(P\I\para P\II)$ and $D^{[\alpha_2]}(P\I\para P\II)$
(as well as the divergences $D^{(\alpha_1)}(P\I\para P\II)$ and $D^{(\alpha_2)}(P\I\para P\II)$)
is close to $1$
for $P\I$ and $P\II$ that are close to each other
(in the sense of $\beta\I/\beta\II\approx1$).
\end{remark*}

The main result of this section is the following non-asymptotic version
of Theorems \ref{thm:criterion} and \ref{thm:general-criterion}.
\begin{theorem}\label{thm:non-asymptotic}
  In the competitive testing protocol:
  \begin{enumerate}
  \item\label{it:na1}
    For any $\alpha\in\bbbr$, $\alpha\notin\{-1,1\}$,
    the Sceptics have a joint strategy guaranteeing that, for all $N$,
    \begin{equation}\label{eq:small-alpha}
      \frac{2}{1+\alpha}
      \ln\K\I_N
      +
      \frac{2}{1-\alpha}
      \ln\K\II_N
      =
      \sum_{n=1}^{N}
      D^{[\alpha]}
      \left(
        P\I_n\para P\II_n
      \right).
    \end{equation}
  \item\label{it:na2}
    For any $\alpha\in(-\infty,-1)$,
    Sceptic I has a strategy guaranteeing,
    for all $N$,
    \begin{equation}\label{eq:big-alpha}
      \frac{2}{1+\alpha}
      \ln\K\I_N
      +
      \frac{2}{1-\alpha}
      \ln\K\II_N
      \le
      \sum_{n=1}^{N}
      D^{[\alpha]}
      \left(
        P\I_n\para P\II_n
      \right).
    \end{equation}
  \end{enumerate}
\end{theorem}
We regard (\ref{eq:small-alpha}) and (\ref{eq:big-alpha}) to be true
if their left-hand side is an indefinite expression of the form $\infty-\infty$.

Part \ref{it:na1} of Theorem \ref{thm:non-asymptotic}
is analogous to Part \ref{it:c1} of Theorem \ref{thm:criterion}.
We will only be interested in the inequality ``$\ge$'' of (\ref{eq:small-alpha})
for $\alpha\in(-1,1)$
(for such $\alpha$s both coefficients $\frac{2}{1+\alpha}$ and $\frac{2}{1-\alpha}$
are positive).
By (\ref{eq:inequalities}) this inequality then implies
\begin{equation*}
  \frac{2}{1+\alpha}
  \ln\K\I_N
  +
  \frac{2}{1-\alpha}
  \ln\K\II_N
  \ge
  \sum_{n=1}^{N}
  D^{(\alpha)}
  \left(
    P\I_n\para P\II_n
  \right),
\end{equation*}
which is a precise quantification of Part \ref{it:c1}
of Theorem \ref{thm:criterion}.

Part \ref{it:na2} of Theorem \ref{thm:non-asymptotic}
greatly strengthens the inequality ``$\le$'' of (\ref{eq:small-alpha})
in the case $\alpha<-1$.
Not only can this inequality be attained
when the Sceptics collaborate with each other,
but Sceptic I alone can enforce it,
even if Sceptic II plays against him.
It is close to being a quantification of Part \ref{it:na2} of Theorem \ref{thm:criterion}.
There is, however,
an essential difference between Part \ref{it:na2} of Theorem \ref{thm:non-asymptotic}
and Part \ref{it:c2} of Theorem \ref{thm:criterion}:
$\alpha<-1$ in the former and $\alpha\in(-1,1)$ in the latter.

By (\ref{eq:inequalities}),
inequality (\ref{eq:big-alpha})
will continue to hold if $D^{[\alpha]}$ is replaced by $D^{(\alpha)}$.
An important special case is where $\alpha=-3$
(considered in \citealt{\VovkDAN}, Theorem 1);
the $(-3)$-divergence becomes the \emph{$\chi^2$ distance}
\begin{equation*}
  D^{(-3)}
  \left(
    P\I\para P\II
  \right)
  =
  \frac12
  \int_{\Omega}
    \frac{(\beta\I(\omega) - \beta\II(\omega))^2}{\beta\II(\omega)}
  Q(\dd\omega)
\end{equation*}
(in the notation of (\ref{eq:divergence}) and assuming $\beta\II>0$;
there is no coefficient $\frac12$ in \citealt{\VovkDAN}).

In the rest of this section we will prove Theorem \ref{thm:non-asymptotic}
mainly following \citet{\VovkDAN} and \citet{fujiwara:2007}.
There are, however, two important differences.
First, our argument will be much more precise
as compared to the $O(1)$ accuracy of the algorithmic theory of randomness.
Second, we will pay careful attention
to the ``exceptional'' cases where $\beta\I_n=0$ or $\beta\II_n=0$;
this corresponds to getting rid of the assumption of local absolute continuity
in measure-theoretic probability
(accomplished by \citealt{pukelsheim:1986}).

\subsection{Proof of Part \ref{it:na1} of Theorem \ref{thm:non-asymptotic}}


Let Sceptic I play the strategy
\begin{multline}\label{eq:sceptic-1}
  f\I_n
  :=
  \frac
  {
    \left(
      \beta\II_n / \beta\I_n
    \right)^{\frac{1+\alpha}{2}}
  }
  {
    \int
      \left(
        \beta\I_n
      \right)^{\frac{1-\alpha}{2}}
      \left(
        \beta\II_n
      \right)^{\frac{1+\alpha}{2}}
    \dd Q_n
  }\\
  =
  \left(
    \frac{\beta\II_n}{\beta\I_n}
  \right)^{\frac{1+\alpha}{2}}
  \exp
  \left(
    \frac{1-\alpha^2}{4}
    D^{[\alpha]}
    \left(
      P\I_n\para P\II_n
    \right)
  \right)
\end{multline}
(the Hellinger integral in the first denominator is just the normalizing constant)
and Sceptic II play the strategy
\begin{multline}\label{eq:sceptic-2}
  f\II_n
  :=
  \frac
  {
    \left(
      \beta\I_n / \beta\II_n
    \right)^{\frac{1-\alpha}{2}}
  }
  {
    \int
      \left(
        \beta\I_n
      \right)^{\frac{1-\alpha}{2}}
      \left(
        \beta\II_n
      \right)^{\frac{1+\alpha}{2}}
    \dd Q_n
  }\\
  =
  \left(
    \frac{\beta\I_n}{\beta\II_n}
  \right)^{\frac{1-\alpha}{2}}
  \exp
  \left(
    \frac{1-\alpha^2}{4}
    D^{[\alpha]}
    \left(
      P\I_n\para P\II_n
    \right)
  \right).
\end{multline}
From
\ifnotFULL
\begin{multline*}
  \left(
    f\I_n
  \right)^{\frac{2}{1+\alpha}}
  \left(
    f\II_n
  \right)^{\frac{2}{1-\alpha}}\\
  =
  \exp
  \left(
    \frac{1-\alpha}{2}
    D^{[\alpha]}
    \left(
      P\I_n\para P\II_n
    \right)
    +
    \frac{1+\alpha}{2}
    D^{[\alpha]}
    \left(
      P\I_n\para P\II_n
    \right)
  \right)\\
  =
  \exp
  \left(
    D^{[\alpha]}
    \left(
      P\I_n\para P\II_n
    \right)
  \right)
\end{multline*}
\fi
we now obtain (\ref{eq:small-alpha}).

Let us now look more carefully at the case
where some of the denominators in (\ref{eq:sceptic-1}) or (\ref{eq:sceptic-2})
are zero
and so the above argument is not applicable directly.
If the Hellinger integral (\ref{eq:Hellinger-integral-1}) at time $n$,
\begin{equation}\label{eq:Hellinger-integral-2}
  \int_{\Omega}
    (\beta\I_n(\omega))^{\frac{1-\alpha}{2}}
    (\beta\II_n(\omega))^{\frac{1+\alpha}{2}}
  Q(\dd\omega),
\end{equation}
is zero,
the probability measures $P\I_n$ and $P\II_n$ are mutually singular.
Choose a local event $E$ such that $P\I_n(E)=P\II_n(\Omega\setminus E)=0$.
If the Sceptics choose
\begin{equation*}
  f\I_n(\omega)
  :=
  \begin{cases}
    \infty & \text{if $\omega\in E$}\\
    1 & \text{otherwise},
  \end{cases}
  \quad
  f\II_n(\omega)
  :=
  \begin{cases}
    1 & \text{if $\omega\in E$}\\
    \infty & \text{otherwise},
  \end{cases}
\end{equation*}
(\ref{eq:small-alpha}) will be guaranteed to hold:
both sides will be $\infty$.

Let us now suppose that the Hellinger integral (\ref{eq:Hellinger-integral-2})
is non-zero.
In (\ref{eq:sceptic-1}) and (\ref{eq:sceptic-2}),
we interpret $0/0$ as $1$
(and, of course, $t/0$ as $\infty$ for $t>0$).
As soon as the local event
\begin{equation*}
  \left(
    \beta\I_n=0 \;\&\; \beta\II_n>0
  \right)
  \text{ or }
  \left(
    \beta\I_n>0 \;\&\; \beta\II_n=0
  \right)
\end{equation*}
happens for the first time
(if it ever happens),
the Sceptics stop playing,
in the sense of selecting $f\I_N=f\II_N:=1$ for all $N>n$.
This will make sure that (\ref{eq:small-alpha}) always holds
(in the sense of the convention
introduced after the statement of the theorem).

\subsection{Proof of Part \ref{it:na2} of Theorem \ref{thm:non-asymptotic}}

Fix $\alpha<-1$ and consider two strategies for Sceptic I:
the one he played before,
(\ref{eq:sceptic-1}),
and the strategy
\begin{equation}\label{eq:strategy-2}
  f\I_n
  =
  \frac{\beta\II_n}{\beta\I_n}
  f\II_n.
\end{equation}
Investing a fraction $c\in(0,1)$ of his initial capital of 1
in strategy (\ref{eq:sceptic-1})
and investing the rest, $1-c$, in strategy (\ref{eq:strategy-2}),
Sceptic I achieves a capital of
\begin{equation}\label{eq:overall-capital}
  c
  \prod_{n=1}^N
  \left(
    \frac{\beta\II_n}{\beta\I_n}
  \right)^{\frac{1+\alpha}{2}}
  \exp
  \left(
    \frac{1-\alpha^2}{4}
    \sum_{n=1}^N
    D^{[\alpha]}
    \left(
      P\I_n\para P\II_n
    \right)
  \right)
  +
  (1-c)
  \K\II_N
  \prod_{n=1}^N
  \frac{\beta\II_n}{\beta\I_n}.
\end{equation}
To get rid of the likelihood ratio
$x:=\prod_{n=1}^N(\beta\II_n/\beta\I_n)$,
we bound (\ref{eq:overall-capital}) from below by
\begin{align}
  &\inf_{x>0}
  \left(
    c
    x^{\frac{1+\alpha}{2}}
    \exp
    \left(
      \frac{1-\alpha^2}{4}
      \sum_{n=1}^N
      D^{[\alpha]}
      \left(
        P\I_n\para P\II_n
      \right)
    \right)
    +
    (1-c)
    \K\II_N
    x
  \right)\label{eq:minimal-capital1}\\
  &=
  \left(
    \left(
      \frac{-1-\alpha}{2}
    \right)^{\frac{2}{1-\alpha}}
    +
    \left(
      \frac{2}{-1-\alpha}
    \right)^{-\frac{1+\alpha}{1-\alpha}}
  \right)
  c^{\frac{2}{1-\alpha}}
  (1-c)^{-\frac{1+\alpha}{1-\alpha}}\label{eq:minimal-capital2}\\
  &\times
  \exp
  \left(
    \frac{1+\alpha}{2}
    \sum_{n=1}^N
    D^{[\alpha]}
    \left(
      P\I_n\para P\II_n
    \right)
  \right)
  \left(
    \K\II_N
  \right)^{-\frac{1+\alpha}{1-\alpha}}.\label{eq:minimal-capital3}
\end{align}
To optimize this lower bound,
we find $c$ such that
\begin{equation*}
  c^{\frac{2}{1-\alpha}}
  (1-c)^{-\frac{1+\alpha}{1-\alpha}}
  \to
  \max,
\end{equation*}
which gives
\begin{equation*}
  c
  =
  \frac{2}{1-\alpha}.
\end{equation*}
For this value of $c$,
the expression on line (\ref{eq:minimal-capital2}) equals
\begin{equation*}
  \left(
    \left(
      \frac{-1-\alpha}{2}
    \right)^{\frac{2}{1-\alpha}}
    +
    \left(
      \frac{2}{-1-\alpha}
    \right)^{-\frac{1+\alpha}{1-\alpha}}
  \right)
  \left(
    \frac{2}{1-\alpha}
  \right)^{\frac{2}{1-\alpha}}
  \left(
    \frac{-1-\alpha}{1-\alpha}
  \right)^{-\frac{1+\alpha}{1-\alpha}}
  =
  1.
\end{equation*}
In combination with (\ref{eq:minimal-capital3}) this gives
\begin{equation*}
  \K\I_N
  \ge
  \exp
  \left(
    \frac{1+\alpha}{2}
    \sum_{n=1}^N
    D^{[\alpha]}
    \left(
      P\I_n\para P\II_n
    \right)
  \right)
  \left(
    \K\II_N
  \right)^{-\frac{1+\alpha}{1-\alpha}}.
\end{equation*}
Taking the logarithm of both sides and regrouping,
we obtain (\ref{eq:big-alpha}).

It remains to consider the exceptional cases.
If the Hellinger integral (\ref{eq:Hellinger-integral-2}) is infinite
(this happens when $P\I_n$ is not absolutely continuous w.r.\ to $P\II_n$),
the right-hand side of (\ref{eq:big-alpha}) is also infinite,
and there is nothing to prove.
Suppose (\ref{eq:Hellinger-integral-2}) is finite.
When $\beta\I_n=0$,
we will interpret the expressions (\ref{eq:sceptic-1}) and (\ref{eq:strategy-2})
as $\infty$
(this will not affect $\int f\I_n\dd P\I_n$;
in principle,
it is now possible that $\int f\I_n\dd P\I_n<1$,
but this can only hurt Sceptic I).
When $\beta\I_n>0\;\&\;\beta\II_n=0$,
we, naturally, interpret (\ref{eq:sceptic-1}) as $\infty$.
In both cases Sceptic I's capital becomes infinite
when $\beta\I_n=0$ or $\beta\II_n=0$,
and (\ref{eq:big-alpha}) still holds.

\section{Proof of Theorems \ref{thm:criterion} and \ref{thm:general-criterion}}
\label{sec:criterion-proof}

Part \ref{it:c1} of Theorems \ref{thm:criterion} and \ref{thm:general-criterion}
immediately follows
from Part \ref{it:na1} of Theorem \ref{thm:non-asymptotic}
(in the case of Theorem \ref{thm:general-criterion},
the Forecasters' moves $f\I_n$ and $f\II_n$ have to be slightly redefined
by setting them to $\infty$ on $E\I$ and $E\II$, respectively).
Therefore, in this section we will only prove Part \ref{it:c2}
of Theorems \ref{thm:criterion} and \ref{thm:general-criterion}.
Instead of deducing this result from Part \ref{it:na2} of Theorem \ref{thm:non-asymptotic}
(as was done in \citealt{\VovkDAN} and \citealt{fujiwara:2007}),
we will prove it using methods of the theory of martingales
and adapting the proof given in Shiryaev (\citeyear{shiryaev:1996}, Theorem VII.6.4).

The following lemma from \citet{fujiwara:2007}
shows that all $\alpha$-divergences, $\alpha\in(-1,1)$,
coincide to within a constant factor.
We will be using the words ``increasing'' and ``decreasing''
in the wide sense
(e.g., a constant function qualifies as both increasing and decreasing).
\begin{lemma}
  Let $P',P''\in\PPP(\Omega)$.
  The function $(1-\alpha)D^{(\alpha)}(P',P'')$
  is decreasing in $\alpha\in(-1,1)$.
  The function $(1+\alpha)D^{(\alpha)}(P',P'')$
  is increasing in $\alpha\in(-1,1)$.
\end{lemma}
\begin{proof}
  See the proof of Lemma 10 in Fujiwara (\citeyear{fujiwara:2007}, Appendix B).
  \ifnotJOURNAL
    \qedtext
  \fi
\end{proof}
Therefore, we could restrict ourselves only to the Hellinger distance.
We will, however, consider the general case
(as \citealt{fujiwara:2007}).

We will be concerned with the following slight modification of the testing protocol.

\medskip

\noindent
\textsc{Semimartingale protocol (multiplicative representation)}

\noindent
\textbf{Players:} Reality, Forecaster, Sceptic

\noindent
\textbf{Protocol:}

\parshape=8
\IndentI   \WidthI
\IndentI   \WidthI
\IndentII  \WidthII
\IndentII  \WidthII
\IndentII  \WidthII
\IndentII  \WidthII
\IndentII  \WidthII
\IndentI   \WidthI
\noindent
$\K_0 := 1$.\\
FOR $n=1,2,\dots$:\\
  Forecaster announces $P_n\in\PPP(\Omega)$.\\
  Reality announces measurable $\xi_n:\Omega\to\bbbr$.\\
  Sceptic announces $f_n:\Omega\to[0,\infty]$
	such that $\int f_n \dd P_n=1$.\\
  Reality announces $\omega_n\in\Omega$.\\
  $\K_n := \K_{n-1} f_n (\omega_n)$.\\
END FOR

\medskip

\noindent
The \emph{martingale protocol}
(resp.\ \emph{submartingale protocol}, \emph{supermartingale protocol})
differs from the semimartingale protocol
in that Reality is required to ensure
that, for all $n$, the function $\xi_n$ is $P_n$-integrable
and $\int\xi_n \dd P_n$ is zero
(resp.\ nonnegative, nonpositive).

An \emph{event} is a property of the play $(P_n,\xi_n,f_n,\omega_n)_{n=1}^{\infty}$.
We say that Sceptic \emph{can force} an event $E$ if he has a strategy
guaranteeing that either $E$ holds or $\K_n\to\infty$ as $n\to\infty$.
Lemma \ref{lem:equivalence} shows
that replacing $\K_n\to\infty$ by $\limsup_n\K_n=\infty$
gives an equivalent definition.

Another representation of the semimartingale protocol is:

\medskip

\noindent
\textsc{Semimartingale protocol (additive representation)}

\noindent
\textbf{Players:} Reality, Forecaster, Sceptic

\noindent
\textbf{Protocol:}

\parshape=8
\IndentI   \WidthI
\IndentI   \WidthI
\IndentII  \WidthII
\IndentII  \WidthII
\IndentII  \WidthII
\IndentII  \WidthII
\IndentII  \WidthII
\IndentI   \WidthI
\noindent
$\K_0 := 1$.\\
FOR $n=1,2,\dots$:\\
  Forecaster announces $P_n\in\PPP(\Omega)$.\\
  Reality announces measurable $\xi_n:\Omega\to\bbbr$.\\
  Sceptic announces $g_n:\Omega\to[-\K_{n-1},\infty]$
	such that $\int g_n \dd P_n=0$.\\
  Reality announces $\omega_n\in\Omega$.\\
  $\K_n := \K_{n-1} + g_n (\omega_n)$.\\
END FOR

\medskip

\noindent
The correspondence between the two representations of the semimartingale protocol
is given by $g_n=(f_n-1)\K_{n-1}$.
We will switch at will between the two representations.

If Sceptic follows a strategy $\SSS$,
we will let $\K_n^{\SSS}$ stand for his capital at the end of step $n$
(as a function of Forecaster's and Reality's moves).
The additive representation makes it obvious that Sceptic's strategies can be mixed:
\begin{lemma}\label{lem:mix}
  If $\SSS_1,\SSS_2,\ldots$ is a sequence of strategies for Sceptic
  and $p_1,p_2,\ldots$ is a sequence of positive weights summing to $1$,
  there is a ``master'' strategy $\SSS$ for Sceptic ensuring
  \begin{equation*}
    \K^{\SSS}_n
    =
    \sum_{k=1}^{\infty}
    p_k
    \K^{\SSS_k}_n.
  \end{equation*}
\end{lemma}
\begin{proof}
  It suffices for Sceptic to set
  $g_n:=\sum_{k=1}^{\infty} p_k g_{k,n}$
  at each step $n$,
  where $g_{k,n}$ is the move recommended by $\SSS_k$.
  \ifnotJOURNAL
    \qedtext
  \fi
\end{proof}

\begin{lemma}\label{lem:martingale}
  In the martingale protocol,
  Sceptic can force
  \begin{equation*}
    \sum_{n=1}^{\infty}
    \int
    \xi_n^2
    \dd P_n
    <
    \infty
    \Longrightarrow
    \sup_N
    \sum_{n=1}^N
    \xi_n(\omega_n)
    <
    \infty.
  \end{equation*}
\end{lemma}
\begin{proof}
  It is easy to see that for each $C>0$
  there is a strategy, say $\SSS_C$, for Sceptic leading to capital
  \begin{equation*}
    \K_N^{\SSS_C}
    =
    \begin{cases}
      1
      +
      \frac1C
      \left(
        \left(
          \sum_{n=1}^N \xi_n(\omega_n)
        \right)^2
        -
        \sum_{n=1}^N
        \int \xi_n^2 \dd P_n
      \right) 
      & \text{if $\sum_{n=1}^N \int \xi_n^2 \dd P_n \le C$}\\
      \K_{N-1}^{\SSS_C} & \text{otherwise},
    \end{cases}
  \end{equation*}
  for all $N=1,2,\ldots$\,.
  It remains to mix all $\SSS_C$, $C=1,2,\ldots$,
  according to Lemma \ref{lem:mix}
  (with arbitrary positive weights).
  \ifnotJOURNAL
    \qedtext
  \fi
\end{proof}

\begin{lemma}\label{lem:submartingale}
  In the submartingale protocol,
  Sceptic can force
  \begin{equation*}
    \sum_{n=1}^{\infty}
    \left(
      \int
      \xi_n
      \dd P_n
      +
      \int
      \xi_n^2
      \dd P_n
    \right)
    <
    \infty
    \Longrightarrow
    \sup_N
    \sum_{n=1}^N
    \xi_n(\omega_n)
    <
    \infty.
  \end{equation*}
\end{lemma}
\begin{proof}
  It suffices to apply Lemma \ref{lem:martingale} to
  $\tilde{\xi}_n:=\xi_n-\int\xi_n\dd P_n$
  (notice that $\int\tilde{\xi}_n^2\dd P_n\le\int\xi_n^2\dd P_n$).
  \ifnotJOURNAL
    \qedtext
  \fi
\end{proof}

Let $\III_E$ stand for the indicator of local event $E\subseteq\Omega$:
\begin{equation*}
  \III_E(\omega)
  :=
  \begin{cases}
    1 & \text{if $\omega\in E$}\\
    0 & \text{otherwise}.
  \end{cases}
\end{equation*}
The following lemma is a version of the Borel--Cantelli--L\'evy lemma.
\begin{lemma}\label{lem:borel-cantelli}
  In the submartingale protocol,
  Sceptic can force
  \begin{equation*}
    \left(
      \forall n:\xi_n=\III_{E_n}
      \;\&\;
      \sum_{n=1}^{\infty}
      P_n(E_n)
      <
      \infty
    \right)
    \Longrightarrow
    \left(
      \omega_n\in E_n
      \text{ for finitely many $n$}
    \right).
  \end{equation*}
\end{lemma}
\begin{proof}
  This is a special case of Lemma \ref{lem:submartingale}.
  \ifnotJOURNAL
    \qedtext
  \fi
\end{proof}
Our application of the Borel--Cantelli--L\'evy lemma
will be made possible by the following lemma,
which we state using the notation introduced earlier in (\ref{eq:divergence}).
\begin{lemma}\label{lem:inequality1}
  For each $\alpha\in(-1,1)$ there exists a constant $C=C(\alpha)$
  such that, for all $P\I$ and $P\II$,
  \begin{equation*}
    P\I
    \left\{
      \beta\I > e\beta\II
    \right\}
    \le
    C
    D^{(\alpha)}
    \left(
      P\I\para P\II
    \right).
  \end{equation*}
\end{lemma}
\begin{proof}
  Let $E$ stand for the event
  $
    \left\{
      \beta\I > e\beta\II
    \right\}
  $.
  Since
  \begin{multline*}
    D^{(\alpha)}
    \left(
      P\I\para P\II
    \right)
    =
    \frac{4}{1-\alpha^2}
    \left(
      1
      -
      \int_{\Omega}
        (\beta\I(\omega))^{\frac{1-\alpha}{2}}
        (\beta\II(\omega))^{\frac{1+\alpha}{2}}
      Q(\dd\omega)
    \right)\\
    =
    \frac{4}{1-\alpha^2}
    \int_{\Omega}
      \frac{1-\alpha}{2}
      \beta\I(\omega)
      +
      \frac{1+\alpha}{2}
      \beta\II(\omega)
      -
      (\beta\I(\omega))^{\frac{1-\alpha}{2}}
      (\beta\II(\omega))^{\frac{1+\alpha}{2}}
    Q(\dd\omega)\\
    \ge
    \frac{4}{1-\alpha^2}
    \int_E
      \frac{1-\alpha}{2}
      \beta\I(\omega)
      +
      \frac{1+\alpha}{2}
      \beta\II(\omega)
      -
      (\beta\I(\omega))^{\frac{1-\alpha}{2}}
      (\beta\II(\omega))^{\frac{1+\alpha}{2}}
    Q(\dd\omega)\\
    \ge
    \frac{4}{1-\alpha^2}
    \int_E
      \frac{1-\alpha}{2}
      \beta\I(\omega)
      +
      \frac{1+\alpha}{2e}
      \beta\I(\omega)
      -
      \left(\beta\I(\omega)\right)^{\frac{1-\alpha}{2}}
      \left(\frac{\beta\I(\omega)}{e}\right)^{\frac{1+\alpha}{2}}
    Q(\dd\omega)\\
    =
    \frac{4}{1-\alpha^2}
    \left(
      \frac{1-\alpha}{2}
      +
      \frac{1+\alpha}{2e}
      -
      e^{-\frac{1+\alpha}{2}}
    \right)
    P\I(E)
  \end{multline*}
  (the first inequality follows from the fact
  that the geometric mean never exceeds the arithmetic mean,
  and the second inequality uses the fact that for each $\beta>0$ the function
  \begin{equation*}
    \frac{1-\alpha}{2}
    \beta
    +
    \frac{1+\alpha}{2}
    x
    -
    \beta^{\frac{1-\alpha}{2}}
    x^{\frac{1+\alpha}{2}}
  \end{equation*}
  is decreasing in $x\in[0,\beta/e]$,
  which can be checked by differentiation),
  we can set
  \begin{equation*}
    C
    :=
    \frac{1-\alpha^2}{4}
    \left(
      \frac{1-\alpha}{2}
      +
      \frac{1+\alpha}{2e}
      -
      e^{-\frac{1+\alpha}{2}}
    \right)^{-1}
    >
    0
  \end{equation*}
  (the expression in the parentheses is a positive function of $\alpha\in(-1,1)$
  since it takes value $0$ at $\alpha=-1$ and $\alpha=1$
  and the function is strictly concave).
  \ifnotJOURNAL
    \qedtext
  \fi
\end{proof}

We will also need the following elementary inequality,
where $U:\bbbr\to\bbbr$ is the truncation function
\begin{equation*}
  U(x)
  :=
  \begin{cases}
    x & \text{if $x\le1$}\\
    1 & \text{otherwise}.
  \end{cases}
\end{equation*}
\begin{lemma}\label{lem:inequality2}
  For each $\gamma\in(0,1)$ there exists $B>1$ such that,
  for all $x>0$,
  \begin{equation}\label{eq:inequality2}
    xU(\ln x) +xU^2(\ln x)
    \le
    B(x-1)
    +
    \frac{B-1}{\gamma}
    \left(
      1-x^{\gamma}
    \right).
  \end{equation}
\end{lemma}
\begin{proof}
  Let us first consider the case $x\le e$.
  To see that
  \begin{equation*}
    x\ln x +x\ln^2 x
    \le
    B(x-1)
    +
    \frac{B-1}{\gamma}
    \left(
      1-x^{\gamma}
    \right)
  \end{equation*}
  notice that the values and the first derivatives of the two sides of this inequality
  coincide when $x=1$
  (in fact, the coefficient $\frac{B-1}{\gamma}$ was chosen to match the derivatives)
  and that the inequality for the second derivatives,
  \begin{equation*}
    \frac{3+2\ln x}{x}
    \le
    \frac{(B-1)(1-\gamma)}{x^{2-\gamma}},
  \end{equation*}
  holds when $B$ is sufficiently large
  (namely, when $B\ge5e^{1-\gamma}/(1-\gamma)+1$).

  In the case $x\ge e$,
  the inequality becomes
  \begin{equation*}
    2x
    \le
    B(x-1)
    +
    \frac{B-1}{\gamma}
    \left(
      1-x^{\gamma}
    \right);
  \end{equation*}
  since it is true for $x=e$
  (by the previous paragraph),
  it suffices to make sure that the inequality
  between the derivatives of the two sides holds:
  \begin{equation*}
    2
    \le
    B
    -
    (B-1)
    x^{\gamma-1}.
  \end{equation*}
  This can be achieved by making
  $B\ge(2e-e^{\gamma})/(e-e^{\gamma})>2$.
  \ifnotJOURNAL
    \qedtext
  \fi
\end{proof}

Now we have all we need to prove Part \ref{it:c2}
of Theorem \ref{thm:general-criterion}.
Let us first assume that the functions $\beta\I_n$ and $\beta\II_n$
are always positive
and that $E\I_n=E\II_n=\emptyset$, for all $n$.
Our goal is to prove that Sceptic I can force
\begin{equation*}
  \sum_{n=1}^{\infty}
  D^{(\alpha)}
  \left(
    P\I_n \para P\II_n
  \right)
  <
  \infty
  \Longrightarrow
  \liminf_{n\to\infty}\K\II_n<\infty.
\end{equation*}
Substituting
\begin{equation*}
  \gamma:=\frac{1-\alpha}{2},
  \quad
  x
  :=
  \frac{\beta\I_n}{\beta\II_n}
\end{equation*}
in (\ref{eq:inequality2}),
multiplying by $\beta\II_n$, integrating over $Q_n$,
and summing over $n=1,2,\ldots$,
we obtain
\begin{multline}\label{eq:key}
  \sum_{n=1}^{\infty}
  \int
  \beta\I_n
  U
  \left(
    \ln\frac{\beta\I_n}{\beta\II_n}
  \right)
  +
  \beta\I_n
  U^2
  \left(
    \ln\frac{\beta\I_n}{\beta\II_n}
  \right)
  \dd Q_n\\
  \le
  \frac{(B-1)(1+\alpha)}{2}
  \sum_{n=1}^{\infty}
  D^{(\alpha)}
  \left(
    P\I_n \para P\II_n
  \right).
\end{multline}
Let us combine this inequality with:
\begin{itemize}
\item
  Lemma \ref{lem:submartingale}
  applied to Sceptic I and $\xi_n:=U\left(\ln\frac{\beta\I_n}{\beta\II_n}\right)$.
  It is applicable because the inequality $xU(\ln x)\ge x-1$,
  valid for all $x>0$,
  implies
  \begin{equation*}
    \int\xi_n \dd P\I_n
    =
    \int
      \frac{\beta\I_n}{\beta\II_n}
      U\left(\ln\frac{\beta\I_n}{\beta\II_n}\right)
      \beta\II_n
    \dd Q_n
    \ge
    \int
      \left(
        \frac{\beta\I_n}{\beta\II_n}
        -
        1
      \right)
      \beta\II_n
    \dd Q_n
    =
    0
  \end{equation*}
  (and $\xi_n$ is $P\I_n$-integrable since $\xi_n\le1$).
\item
  Lemma \ref{lem:borel-cantelli}
  applied to Sceptic I and $E_n:=\{\beta\I_n>e\beta\II_n\}$.
  It will be applicable by Lemma \ref{lem:inequality1}.
\end{itemize}
We can now see that it suffices to prove that Sceptic I can force
\begin{multline}\label{eq:forcing1}
  \left(
    \sup_N
    \sum_{n=1}^N
    U
    \left(
      \ln\frac{\beta\I_n(\omega_n)}{\beta\II_n(\omega_n)}
    \right)
    <
    \infty
    \;\&\;
    \beta\I_n(\omega_n)\le e\beta\II_n(\omega_n)
    \text{ from some $n$ on}
  \right)\\
  \Longrightarrow
  \liminf_{n\to\infty}\K\II_n<\infty.
\end{multline}
Forcing (\ref{eq:forcing1}) can be achieved by forcing
\begin{equation}\label{eq:forcing2}
  \sup_N
  \sum_{n=1}^N
  \ln\frac{\beta\I_n(\omega_n)}{\beta\II_n(\omega_n)}
  <
  \infty
  \Longrightarrow
  \liminf_{n\to\infty}\K\II_n<\infty,
\end{equation}
and the latter can be done with the strategy
\begin{equation*}
  f\I_n
  :=
  \frac{\beta\II_n}{\beta\I_n}
  f\II_n.
\end{equation*}

It remains to get rid of the assumption
$\beta\I_n>0$, $\beta\II_n>0$, $E\I_n=E\II_n=\emptyset$.
Our argument so far remains valid
if the observation space $\Omega=\Omega_n$
is allowed to depend on $n$
(it can be chosen by Sceptic I's opponents
at any time prior to his move).
In particular, we can set $\Omega_n:=\Omega\setminus(E\I_n\cup E\II_n)$
and assume, without loss of generality,
that $\beta\I_n>0$ and $\beta\II_n>0$ on $\Omega_n$.
This proves Part 2 of Theorem \ref{thm:general-criterion}.

To deduce Part 2 of Theorem \ref{thm:criterion}
from Part 2 of Theorem \ref{thm:general-criterion},
notice that, for all $n$,
$P\I_n(E\II_n)=0$ (as $P\I_n\ll P\II_n$).
Therefore,
$P\I_n(E\I_n\cup E\II_n)=0$,
and mixing (in the sense of Lemma \ref{lem:mix})
any of the strategies for Sceptic I
whose existence is asserted in Part 2 of Theorem \ref{thm:general-criterion}
with the strategy
\begin{equation*}
  f\I_n(\omega)
  :=
  \begin{cases}
    \infty & \text{if $\omega\in E\I_n\cup E\II_n$}\\
    1 & \text{otherwise}
  \end{cases}
\end{equation*}
we obtain a strategy for Sceptic I
satisfying the condition of Part 2 of Theorem \ref{thm:criterion}.

\begin{remark*}
  Part \ref{it:na2} of Theorem \ref{thm:non-asymptotic}
  gives a precise lower bound for the rate of growth of Sceptic I's capital
  in terms of the rate of growth of Sceptic II's capital
  and the cumulative divergence between the Forecasters' predictions.
  Unfortunately,
  our proof of Part \ref{it:c2} of Theorems \ref{thm:criterion} and \ref{thm:general-criterion}
  does not give such a bound;
  the step that prevents us from obtaining such a bound
  is replacing (\ref{eq:forcing1}) with (\ref{eq:forcing2})
  (another such step would have been the use of Doob's martingale convergence theorem,
  as in \citealt{shiryaev:1996},
  but this proof avoids it).
  It remains an open problem whether such a bound exists.
\end{remark*}

\section{The rate of growth of Sceptic II's capital
  in terms of the Kullback--Leibler divergence}
\label{sec:rate}

Suppose Forecaster I is reliable
and Sceptic I invests a part of his capital in strategies
whose existence is asserted in Theorem \ref{thm:non-asymptotic}.
That theorem then gives rather precise bounds for the achievable rate of growth
of $\K\II_n$:
Sceptic II can achieve
\begin{equation}\label{eq:can}
  \ln\K\II_N
  \ge
  \frac{1-\alpha}{2}
  \sum_{n=1}^{N}
  D^{[\alpha]}
  \left(
    P\I_n\para P\II_n
  \right)
  -
  O(1)
\end{equation}
but cannot achieve more than that:
\begin{equation}\label{eq:cannot}
  \ln\K\II_N
  \le
  \frac{1-\alpha}{2}
  \sum_{n=1}^{N}
  D^{[\alpha]}
  \left(
    P\I_n\para P\II_n
  \right)
  +
  O(1).
\end{equation}
The problem with this is that (\ref{eq:can}) and (\ref{eq:cannot})
refer to different $\alpha$s:
$\alpha\in(-1,1)$ in (\ref{eq:can})
and $\alpha<-1$ in (\ref{eq:cannot}).
In this section we will derive similar (albeit cruder) bounds
for the same $\alpha$-divergence,
namely for $\alpha=-1$,
one of the two cases that have been excluded so far.
The \emph{$(-1)$-divergence}, or the \emph{Kullback--Leibler divergence},
is defined by
\begin{equation*}
  D^{(-1)}
  \left(
    P\I\para P\II
  \right)
  :=
  \int_{\Omega}
    \ln\frac{\beta\I(\omega)}{\beta\II(\omega)}
  P\I(\dd\omega),
\end{equation*}
using the notation of (\ref{eq:divergence}).
Somewhat related results have been obtained in algorithmic randomness theory
by \citet{solomonoff:1978};
see also \citet{\VovkPPI}, Theorem 2.2.

\begin{remark*}
  It might be tempting to set $\alpha<-1$ in (\ref{eq:small-alpha}).
  This will not lead to any useful bounds because of the possibility $\K\I_n\to0$.
\end{remark*}

For simplicity,
we will again impose the assumption of ``timidity''
(in a much stronger sense than before)
on Forecaster II:
on the given play of the game,
his predictions should stay within a constant factor of the reliable Forecaster I.
More precisely,
Forecaster II is \emph{$c$-timid}, for a constant $c>1$,
if for all $n$ the ratio $\beta\II_n/\beta\I_n$
(with $0/0$ interpreted as $1$)
is bounded above by $c$ and bounded below by $1/c$.
The value of the constant $c$ is not disclosed to the players,
and the strategies for the Sceptics constructed in this section
will never depend on $c$.

Let $x^+$ stand for $\max(x,0)$.
\begin{theorem}\label{thm:growth}
  Let $c>1$.
  There is a constant $C>0$ depending only on $c$ such that:
  \begin{enumerate}
  \item\label{it:growth1}
    The Sceptics have a joint strategy in the competitive testing protocol
    that guarantees,
    starting from $N=3$,
    \begin{equation}\label{eq:thm-growth-1}
      \ln\K\II_N
      \ge
      \sum_{n=1}^N
      D^{(-1)}
      \left(
        P\I_n\para P\II_n
      \right)
      -
      C
      \sqrt{\frac{N}{\ln\ln N}}
      \left(
        \ln^{+}\K\I_N
        +
        \ln\ln N
      \right)
    \end{equation}
    on the plays where Forecaster II is $c$-timid.
  \item\label{it:growth2}
    Sceptic I has a strategy that guarantees,
    starting from $N=3$,
    \begin{equation}\label{eq:thm-growth-2}
      \ln\K\II_N
      \le
      \sum_{n=1}^N
      D^{(-1)}
      \left(
        P\I_n\para P\II_n
      \right)
      +
      C
      \sqrt{\frac{N}{\ln\ln N}}
      \left(
        \ln^{+}\K\I_N
        +
        \ln\ln N
      \right)
    \end{equation}
    on the plays where Forecaster II is $c$-timid.
  \end{enumerate}
\end{theorem}
Before proving Theorem \ref{thm:growth},
we will state a similar result for the case
where the duration of the game, $N$, is known in advance.
(Theorem \ref{thm:growth} is applicable in this case as well,
but it can be made more precise.)
\begin{proposition}\label{prop:growth}
  Let $c>1$.
  There is a constant $C>0$ depending only on $c$ such that:
  \begin{enumerate}
  \item
    For each $N\in\{2,3,\ldots\}$,
    the Sceptics have a joint strategy that guarantees
    \begin{equation}\label{eq:prop-growth-1}
      \ln\K\II_N
      \ge
      \sum_{n=1}^N
      D^{(-1)}
      \left(
        P\I_n\para P\II_n
      \right)
      -
      \left(
        \sqrt{N} - 1
      \right)
      \left(
        \ln\K\I_N
        +
        C
      \right)
    \end{equation}
    on the plays where Forecaster II is $c$-timid.
  \item
    For each $N\in\{1,2,\ldots\}$,
    Sceptic I has a strategy that guarantees
    \begin{equation}\label{eq:prop-growth-2}
      \ln\K\II_N
      \le
      \sum_{n=1}^N
      D^{(-1)}
      \left(
        P\I_n\para P\II_n
      \right)
      +
      \left(
        \sqrt{N} + 1
      \right)
      \left(
        \ln\K\I_N
        +
        C
      \right)
    \end{equation}
    on the plays where Forecaster II is $c$-timid.
  \end{enumerate}
\end{proposition}

The intuition behind Theorem \ref{thm:growth} and Proposition \ref{prop:growth}
is that Sceptic II can achieve the growth rate of his logarithmic capital
close to the growth rate of the cumulative Kullback--Leibler divergence
between $P\I_n$ and $P\II_n$,
but cannot achieve a better growth rate.
Theorem \ref{thm:growth} is related to the law of the iterated logarithm,
in that it gives the accuracy of $O(\sqrt{N\ln\ln N})$
for a reliable Forecaster I,
whereas Proposition \ref{prop:growth} is related to the central limit theorem,
in that it gives the accuracy of $O(\sqrt{N})$.

\subsection{Proof of Part \ref{it:growth1} of Theorem \ref{thm:growth}
  and Proposition \ref{prop:growth}}

Substituting $\alpha:=-1+2\epsilon$, with $\epsilon\in(0,1)$,
in (\ref{eq:small-alpha}) gives
\begin{equation*}
  \frac{1}{\epsilon}
  \ln\K\I_N
  +
  \frac{1}{1-\epsilon}
  \ln\K\II_N
  =
  \sum_{n=1}^{N}
  D^{[-1+2\epsilon]}
  \left(
    P\I_n\para P\II_n
  \right),
\end{equation*}
which can be rewritten as
\begin{equation}\label{eq:growth-intermediate-1}
  \ln\K\II_N
  =
  -
  \frac{1}{\epsilon}
  \sum_{n=1}^{N}
  \ln
  \int
    \left(
      \frac{\beta\II_n}{\beta\I_n}
    \right)^{\epsilon}
  \dd P\I_n
  -
  \frac{1-\epsilon}{\epsilon}
  \ln\K\I_N.
\end{equation}
The inequality $e^x\le1+x+\frac{x^2}{2}e^{x^+}$ gives
\begin{align*}
  \left(
    \frac{\beta\II_n}{\beta\I_n}
  \right)^{\epsilon}
  &=
  \exp
  \left(
    \epsilon
    \ln\frac{\beta\II_n}{\beta\I_n}
  \right)\\
  &\le
  1
  +
  \epsilon
  \ln\frac{\beta\II_n}{\beta\I_n}
  +
  \frac12
  \epsilon^2
  \ln^2\frac{\beta\II_n}{\beta\I_n}
  \exp
  \left(
    \epsilon
    \ln^{+}\frac{\beta\II_n}{\beta\I_n}
  \right)\\
  &\le
  1
  +
  \epsilon
  \ln\frac{\beta\II_n}{\beta\I_n}
  +
  \frac12
  \epsilon^2
  c^{\epsilon}\ln^2c.
\end{align*}
This further implies
\begin{align}
  -
  \frac{1}{\epsilon}
  \sum_{n=1}^{N}
  &\ln
  \int
    \left(
      \frac{\beta\II_n}{\beta\I_n}
    \right)^{\epsilon}
  \dd P\I_n\notag\\
  &\ge
  -
  \frac{1}{\epsilon}
  \sum_{n=1}^{N}
  \ln
  \left(
    1
    -
    \epsilon
    D^{(-1)}
    \left(
      P\I_n\para P\II_n
    \right)
    +
    \frac12
    \epsilon^2
    c^{\epsilon}\ln^2c
  \right)\label{eq:before}\\
  &\ge
  -
  \frac{1}{\epsilon}
  \sum_{n=1}^{N}
  \left(
    -
    \epsilon
    D^{(-1)}
    \left(
      P\I_n\para P\II_n
    \right)
    +
    \frac12
    \epsilon^2
    c^{\epsilon}\ln^2c
  \right)\label{eq:after}\\
  &=
  \sum_{n=1}^{N}
  D^{(-1)}
  \left(
    P\I_n\para P\II_n
  \right)
  -
  \frac12
  N
  \epsilon
  c^{\epsilon}\ln^2c\notag
\end{align}
(the transition from (\ref{eq:before}) to (\ref{eq:after})
uses the inequality $\ln x\le x-1$, valid for $x\ge0$;
the expression in the parentheses in (\ref{eq:before}) is nonnegative
because of its provenance).
Plugging the last inequality into (\ref{eq:growth-intermediate-1}),
we obtain
\begin{equation}\label{eq:growth-intermediate-2}
  \ln\K\II_N
  \ge
  \sum_{n=1}^{N}
  D^{(-1)}
  \left(
    P\I_n\para P\II_n
  \right)
  -
  \frac12
  N
  \epsilon
  c^{\epsilon}\ln^2c
  -
  \frac{1-\epsilon}{\epsilon}
  \ln\K\I_N.
\end{equation}

If $N$ is known in advance (but $c$ and $\K\I_N$ are not),
we can set $\epsilon:=N^{-1/2}$ in (\ref{eq:growth-intermediate-2}),
which gives
\begin{equation*}
  \ln\K\II_N
  \ge
  \sum_{n=1}^{N}
  D^{(-1)}
  \left(
    P\I_n\para P\II_n
  \right)
  -
  \frac{1}{2}
  \sqrt{N}
  c^{N^{-1/2}}\ln^2c
  -
  \left(
    \sqrt{N}-1
  \right)
  \ln\K\I_N;
\end{equation*}
this proves (\ref{eq:prop-growth-1}) with, e.g., $C:=2c\ln^2c$
(in fact, we can take $C$ arbitrarily close to $\frac12\ln^2c$
if we only want (\ref{eq:prop-growth-1}) to hold for sufficiently large $N$).

In the case of Theorem \ref{thm:growth},
where $N$ is unknown,
we will use the discrete form (\citealt{\VovkTVP}, Theorem 1)
of Ville's (\citeyear{ville:1939}) method of proving the law of the iterated logarithm
(however, without worrying about constant factors).
The method is based on the following simple corollary of Lemma \ref{lem:mix}.
If $\SSS_2,\SSS_3,\ldots$ is a sequence of strategies for Sceptic
and $p_2,p_3,\ldots$ is a sequence of positive weights summing to $1$,
there is a strategy $\SSS$ for Sceptic ensuring
$
  \ln\K^{\SSS}_N
  \ge
  \ln\K^{\SSS_k}_N
  +
  \ln p_k
$
for all $k=2,3,\ldots$\,.
In particular, taking $p_k\propto k^{-2}$, $k=2,3,\ldots$,
we obtain
$
  \ln\K^{\SSS}_N
  \ge
  \ln\K^{\SSS_k}_N
  -
  2\ln k
$
for $k\ge2$.

Let $\epsilon_2,\epsilon_3,\ldots$ be a sequence of positive numbers,
to be specified later on.
Since for each $k=2,3,\ldots$
the Sceptics can ensure (\ref{eq:growth-intermediate-2}) for $\epsilon:=\epsilon_k$
(using simple strategies (\ref{eq:sceptic-1})--(\ref{eq:sceptic-2})
depending only on the Forecasters' predictions),
they can also ensure,
for all $k=2,3,\ldots$,
\begin{equation}\label{eq:growth-intermediate-3}
  \ln\K\II_N
  +
  2\ln k
  \ge
  \sum_{n=1}^{N}
  D^{(-1)}
  \left(
    P\I_n\para P\II_n
  \right)
  -
  C_1
  N
  \epsilon_k
  -
  \frac{1-\epsilon_k}{\epsilon_k}
  \left(
    \ln\K\I_N
    +
    2\ln k
  \right)
\end{equation}
where $C_1$ (as well as $C_2$ to $C_6$ below)
is a constant depending only on $c$.
Weakening (\ref{eq:growth-intermediate-3}) to
\begin{equation*}
  \ln\K\II_N
  \ge
  \sum_{n=1}^{N}
  D^{(-1)}
  \left(
    P\I_n\para P\II_n
  \right)
  -
  C_1
  N
  \epsilon_k
  -
  \frac{1}{\epsilon_k}
  \ln\K\I_N
  -
  \frac{2\ln k}{\epsilon_k}
\end{equation*}
and setting
\begin{equation}\label{eq:setting}
  k
  :=
  \lceil\ln N\rceil,
  \quad
  \epsilon_k
  :=
  \sqrt{\frac{\ln k}{e^k}}
\end{equation}
(so that $\epsilon_k$ coincides with $\sqrt{\ln\ln N/N}$
to within a constant factor),
we further obtain
\begin{equation*}
  \ln\K\II_N
  \ge
  \sum_{n=1}^{N}
  D^{(-1)}
  \left(
    P\I_n\para P\II_n
  \right)
  -
  C_2
  \sqrt{N\ln\ln N}
  -
  C_3
  \sqrt{\frac{N}{\ln\ln N}}
  \ln^{+}\K\I_N
\end{equation*}
for $N\ge3$.
This essentially coincides with (\ref{eq:thm-growth-1}).

\begin{remark*}
  The strategies for Sceptic II constructed in our proof of Part \ref{it:growth1}
  of Theorem \ref{thm:growth} and Proposition \ref{prop:growth}
  was somewhat complex,
  especially in the case of Theorem \ref{thm:growth}.
  In the spirit of \citet{solomonoff:1978},
  we could take the simple ``likelihood ratio'' strategy
  $f\II_n:=\beta\I_n/\beta\II_n$.
  The expected value of $\ln f\II_n$ with respect to $P\I_n$
  is $D^{(-1)}(P\I_n\para P\II_n)$,
  and according to standard results of game-theoretic probability
  Sceptic I can become infinitely rich unless
  $
    \ln\K\II_N
    \approx
    \sum_{n=1}^N
    D^{(-1)}
    \left(
      P\I_n\para P\II_n
    \right)
  $.
  The law of the iterated logarithm
  (see, e.g., \citealt{shafer/vovk:2001}, Chapter 5)
  will give a result similar to Part \ref{it:growth1}
  of Theorem \ref{thm:growth},
  and the weak law of large numbers
  (in the form of Proposition 6.1 in \citealt{shafer/vovk:2001})
  or the central limit theorem
  (\citealt{shafer/vovk:2001}, Chapters 6--7)
  will give results similar to Part \ref{it:growth1}
  of Proposition \ref{prop:growth}.
  An advantage of the proofs given in this subsection
  is that they show the dependence of Sceptic II's capital on Sceptic I's capital;
  it remains to be seen whether such explicit dependence
  can be achieved using limit theorems of game-theoretic probability.
\end{remark*}

\subsection{Proof of Part \ref{it:growth2} of Theorem \ref{thm:growth}
  and Proposition \ref{prop:growth}}

Substituting $\alpha:=-1-2\epsilon$, with $\epsilon>0$,
in (\ref{eq:big-alpha}) gives
\begin{equation*}
  -\frac{1}{\epsilon}
  \ln\K\I_N
  +
  \frac{1}{1+\epsilon}
  \ln\K\II_N
  \le
  \sum_{n=1}^{N}
  D^{[-1-2\epsilon]}
  \left(
    P\I_n\para P\II_n
  \right),
\end{equation*}
or, equivalently,
\begin{equation}
  \ln\K\II_N
  \le
  \frac{1}{\epsilon}
  \sum_{n=1}^{N}
  \ln
  \int
  \left(
    \frac{\beta\I_n}{\beta\II_n}
  \right)^{\epsilon}
  \dd P\I_n
  +
  \frac{1+\epsilon}{\epsilon}
  \ln\K\I_N.
\end{equation}
Analogously to the transition
from (\ref{eq:growth-intermediate-1}) to (\ref{eq:growth-intermediate-2})
but now using
\begin{equation*}
  \left(
    \frac{\beta\I_n}{\beta\II_n}
  \right)^{\epsilon}
  \le
  1
  +
  \epsilon
  \ln\frac{\beta\I_n}{\beta\II_n}
  +
  \frac12
  \epsilon^2
  c^{\epsilon}\ln^2c
\end{equation*}
and
\begin{align*}
  \frac{1}{\epsilon}
  \sum_{n=1}^{N}
  \ln
  \int
  \left(
    \frac{\beta\I_n}{\beta\II_n}
  \right)^{\epsilon}
  \dd P\I_n
  &\le
  \frac{1}{\epsilon}
  \sum_{n=1}^{N}
  \ln
  \left(
    1
    +
    \epsilon
    D^{(-1)}
    \left(
      P\I_n\para P\II_n
    \right)
    +
    \frac12
    \epsilon^2
    c^{\epsilon}\ln^2c
  \right)\\
  &\le
  \sum_{n=1}^{N}
  D^{(-1)}
  \left(
    P\I_n\para P\II_n
  \right)
  +
  \frac12
  N
  \epsilon
  c^{\epsilon}\ln^2c,
\end{align*}
we obtain
\begin{equation*}
  \ln\K\II_N
  \le
  \sum_{n=1}^{N}
  D^{(-1)}
  \left(
    P\I_n\para P\II_n
  \right)
  +
  \frac12
  N
  \epsilon
  c^{\epsilon}\ln^2c
  +
  \frac{1+\epsilon}{\epsilon}
  \ln\K\I_N.
\end{equation*}
If $N$ is known in advance,
setting $\epsilon:=N^{-1/2}$\Extra{\ (which implies $\frac{1+\epsilon}{\epsilon}=\sqrt{N}+1$)}
gives
\begin{equation*}
  \ln\K\II_N
  \le
  \sum_{n=1}^{N}
  D^{(-1)}
  \left(
    P\I_n\para P\II_n
  \right)
  +
  \left(
    C_4
    +
    \ln\K\I_N
  \right)
  \left(
    \sqrt{N}+1
  \right)
\end{equation*}
and so completes the proof of Proposition \ref{prop:growth}.
(We can take $C:=c\ln^2c$ in (\ref{eq:prop-growth-2}),
or, if we are interested in (\ref{eq:prop-growth-2}) holding from some $N$ on,
$C\approx\frac12\ln^2c$.)

As for Theorem \ref{thm:growth},
we now have
\begin{equation*}
  \ln\K\II_N
  \le
  \sum_{n=1}^{N}
  D^{(-1)}
  \left(
    P\I_n\para P\II_n
  \right)
  +
  \frac12
  N
  \epsilon_k
  c^{\epsilon_k}
  \ln^2c
  +
  \frac{1+\epsilon_k}{\epsilon_k}
  \left(
    \ln\K\I_N
    +
    2\ln k
  \right).
\end{equation*}
Setting, as before, (\ref{eq:setting}),
we now obtain\Extra{\
$\frac{1+\epsilon_k}{\epsilon_k}=1+\frac{1}{\epsilon_k}\le1+\sqrt{N/\ln\ln N}$ and}
\begin{multline*}
  \ln\K\II_N
  \le
  \sum_{n=1}^{N}
  D^{(-1)}
  \left(
    P\I_n\para P\II_n
  \right)\\
  +
  C_5
  \sqrt{N\ln\ln N}
  +
  C_6
  \left(
    1
    +
    \sqrt{\frac{N}{\ln\ln N}}
  \right)
  \left(
    \ln^{+}\K\I_N
    +
    2\ln\ln N
  \right),
\end{multline*}
which completes the proof of Theorem \ref{thm:growth}.

\section{Criteria of absolute continuity and singularity}
\label{sec:kabanov}

A simple measure-theoretic counterpart of the competitive testing protocol
is the measurable space $\Omega^{\infty}$ equipped with two probability measures,
$\Prob\I,\Prob\II\in\PPP(\Omega^{\infty})$.
The generic element of $\Omega^{\infty}$ will be denoted $\omega_1\omega_2\ldots$\,.
Let $P\I_n$ (resp.\ $P\II_n$) be a regular conditional distribution
of $\omega_n$ given $\omega_1\ldots\omega_{n-1}$
w.r.\ to the probability measure $\Prob\I$ (resp.\ $\Prob\II$).
(For regular conditional distributions to exist
it suffices to assume that $\Omega$ is a Borel space:
see, e.g., \citealt{shiryaev:1996}, Theorem II.7.5.)
The strategies of the two Forecasters are fixed:
Forecaster I is playing $P\I_n$ and Forecaster II is playing $P\II_n$;
therefore,
they cease to be active players in the game.

We will consider the filtration $(\FFF_n)_{n=0}^{\infty}$
where each $\FFF_n$ is generated by $\omega_1,\ldots,\omega_n$
and sometimes write $\FFF$ for $\FFF_{\infty}$.
By a \emph{normalized nonnegative measure-theoretic martingale}
w.r.\ to a probability measure $\Prob$ on $\Omega^{\infty}$
we will mean a martingale $(\xi_n)_{n=0}^{\infty}$
(see, e.g., \citealt{shiryaev:1996}, Chapter VII)
w.r.\ to $\Prob$ and $(\FFF_n)_{n=0}^{\infty}$
such that $\xi_0=1$ and $\xi_n\ge0$ for all $n=1,2,\ldots$;
we will allow $\xi_n$ to take value $\infty$
(of course, with probability zero).
We will usually write $\xi(\omega_1,\ldots,\omega_n)$ for $\xi_n(\omega_1,\omega_2,\ldots)$
and regard $\xi$ as a function on the set $\Omega^*$
of all finite sequences of elements of $\Omega$.
A \emph{normalized nonnegative game-theoretic martingale} w.r.\ to $\Prob$
(either $\Prob\I$ or $\Prob\II$)
is Sceptic's (either Sceptic I's or Sceptic II's, respectively) capital
represented as a function of Reality's moves $\omega_1\ldots\omega_n$.

To state the connection between measure-theoretic and game-theoretic notions
of normalized nonnegative martingale
(in the current measure-theoretic framework),
let us say that two \emph{processes}
(i.e., measurable functions on $\Omega^{*}$)
are \emph{equivalent} (w.r.\ to $\Prob$)
if they coincide $\Prob$-almost surely.
It is customary in measure-theoretic probability
to identify equivalent processes
(although this practice does not carry over to the game-theoretic framework).
It is easy to see that the normalized nonnegative measure-theoretic martingales
are the closure of the measurable normalized nonnegative game-theoretic martingales
w.r.\ to this relation of equivalence.

The following proposition is a special case of the infinitary version of Ville's theorem;
for a proof see, e.g., \citet{shafer/vovk:2001}, Proposition 8.14.
\begin{proposition}\label{prop:Ville-zero}
  Let $E\in\FFF$ and $\Prob$ be either $\Prob\I$ or $\Prob\II$.
  \begin{enumerate}
  \item
    If a normalized nonnegative measure-theoretic martingale
    diverges to infinity when $E$ happens,
    then $\Prob E = 0$.
  \item
    If $\Prob E = 0$,
    then there is a normalized nonnegative measure-theoretic martingale
    that diverges to infinity if $E$ happens.
  \end{enumerate}
\end{proposition}
This proposition is also true
for the measurable normalized nonnegative game-theoretic martingales
and establishes the connection between the game-theoretic and measure-theoretic notions
of a ``null event'':
for example,
an event (in the measure-theoretic sense,
i.e., a measurable subset of $\Omega^{\infty}$) $E$ satisfies $\Prob\I(E)=0$
if and only if Sceptic I has a measurable strategy
that makes him infinitely rich on $E$.

The following is a special case
of the Kabanov--Liptser--Shiryaev (\citeyear{\Kabanov}) criterion
of absolute continuity and singularity.
For simplicity we assume that
$\left.\Prob\I\right|_{\FFF_n}\ll\left.\Prob\II\right|_{\FFF_n}$
for all $n$
(this is the standard assumption of local absolute continuity,
which simplifies measure-theoretic results:
see, e.g., \citealt{jacod/shiryaev:2003-local}, Sect.~IV.2c).
This will ensure the timidity of Forecaster II
(at least after changing the regular conditional distributions
on a set of probability zero, both under $\Prob\I$ and $\Prob\II$).

\begin{corollary}\label{cor:ac-s}
  In the measure-theoretic competitive testing protocol:
  \begin{enumerate}
  \item\label{it:ac-s-1}
    For any $\alpha\in(-1,1)$,
    $\Prob\I\ll\Prob\II$ if and only if (\ref{eq:close}) holds $\Prob\I$-almost surely.
  \item\label{it:ac-s-2}
    For any $\alpha\in(-1,1)$,
    $\Prob\I\perp\Prob\II$ if and only if (\ref{eq:far}) holds $\Prob\I$-almost surely.
  \end{enumerate}
\end{corollary}
\begin{proof}
  We start from Part ``if'' of Part \ref{it:ac-s-1}.
  Suppose (\ref{eq:close}) holds $\Prob\I$-almost surely
  and $E$ is an event such that $\Prob\II(E)=0$;
  our goal is to prove that $\Prob\I(E)=0$.
  Let Sceptic II play a measurable strategy that makes him infinitely rich
  on the event $E$,
  and let Sceptic I play the half-and-half mixture
  (in the sense of Lemma \ref{lem:mix})
  of the following two strategies:
  one of the measurable strategies
  whose existence is guaranteed in Part \ref{it:c2} of Theorem \ref{thm:criterion}
  (cf.\ the remark following this proof)
  and a measurable strategy that makes him infinitely rich
  when the event (\ref{eq:close}) fails to happen.
  It is easy to check that Sceptic I is guaranteed
  to become infinitely rich on the event $E$,
  no matter whether (\ref{eq:close}) holds or not.
  Therefore, indeed $\Prob\I(E)=0$.

  Next we prove Part ``only if'' of Part \ref{it:ac-s-1}.
  Fix measurable strategies $\SSS\I$ and $\SSS\II$ for the Sceptics
  that make at least one of them infinitely rich
  when the event (\ref{eq:close}) fails to happen;
  Part \ref{it:c1} of Theorem \ref{thm:criterion}
  guarantees that such strategies exist.
  Sceptic II will play strategy $\SSS\II$
  whereas Sceptic I will play a mixture of $\SSS\I$ and another strategy.
  Let $E$ be the event that Sceptic II becomes infinitely rich.
  Since $\Prob\II(E)=0$ and $\Prob\I\ll\Prob\II$,
  we have $\Prob\I(E)=0$
  and so Sceptic I has a measurable strategy that makes him infinitely rich on $E$;
  let him play the half-and-half mixture of this strategy and $\SSS\I$.
  This strategy for Sceptic I will guarantee
  his becoming infinitely rich whenever (\ref{eq:close}) fails to happen.

  The proof of Part ``if'' of Part \ref{it:ac-s-2}
  again relies on Part \ref{it:c1} of Theorem \ref{thm:criterion}.
  Define $\SSS\I$ and $\SSS\II$ as before,
  let Sceptic II play $\SSS\II$,
  and let Sceptic I play the half-and-half mixture of $\SSS\I$
  and a measurable strategy that makes him infinitely rich
  when (\ref{eq:far}) fails to happen.
  It is clear that one of the Sceptics becomes infinitely rich
  no matter what happens.
  Therefore, $\Prob\I$ and $\Prob\II$ are mutually singular
  (e.g., take $E$ as the event that Sceptic I becomes infinitely rich;
  then $\Prob\I(E)=0$ and $\Prob\II(\Omega^{\infty}\setminus E)=0$).

  It remains to prove Part ``only if'' of Part \ref{it:ac-s-2}.
  Let $E$ be an event such that $\Prob\II(E)=0$ and $\Prob\I(E)=1$.
  Let Sceptic II play a measurable strategy
  that makes him infinitely rich on $E$
  and let Sceptic I play the half-and-half mixture
  of the following two strategies:
  one of the measurable strategies
  whose existence is guaranteed in Part \ref{it:c2} of Theorem \ref{thm:criterion}
  and a measurable strategy that makes him infinitely rich
  when the event $E$ fails to happen.
  Now if the event (\ref{eq:far}) fails to happen,
  Sceptic I is guaranteed to become infinitely rich,
  regardless of whether $E$ happens.
  This completes the proof.
  \ifnotJOURNAL
    \qedtext
  \fi
\end{proof}

\begin{remark*}
  Notice that to deduce Corollary \ref{cor:ac-s} we need slightly more
  than stated in Theorem \ref{thm:criterion}:
  namely, we need measurable strategies
  for the Sceptics (in the case of Part \ref{it:c1})
  or Sceptic I (in the case of Part \ref{it:c2}).
  It is easy to see that the strategies constructed in Theorem \ref{thm:criterion}
  satisfy this property,
  but it would have been awkward to include their measurability in the statement.
  The strategies are not only measurable but also computable,
  which is a much stronger property.
  Therefore, it would have been more natural to include the strategies' computability
  in the statement of Theorem \ref{thm:criterion};
  however, this would significantly complicate the exposition,
  especially that there are several popular non-equivalent definitions of computability,
  even for real-valued functions of real variable.
  Following \citet{shafer/vovk:2001},
  we dropped any references to the properties of regularity
  of the constructed strategies for the Sceptics
  in the formal statements of our results.
\end{remark*}

To discuss the intuition behind Corollary \ref{cor:ac-s},
let us assume that Forecaster I is reliable
in the following measure-theoretic sense:
we do not expect an event $E$ to happen
if it is chosen in advance and satisfies $\Prob\I(E)=0$.
Then $\Prob\I\ll\Prob\II$ means that Forecaster II
is ``automatically reliable'':
if $E$ is chosen in advance and satisfies $\Prob\II(E)=0$,
we also do not expect it to happen.
Part \ref{it:ac-s-1} of Corollary \ref{cor:ac-s}
says that Forecaster II is automatically reliable
if and only if the Forecasters' predictions are close in the sense of (\ref{eq:close})
holding almost surely w.r.\ to the reliable probability measure $\Prob\I$.

On the other hand,
$\Prob\I\perp\Prob\II$ means that Forecaster II
is ``automatically unreliable'':
we can choose in advance an event $E$
which we expect to happen ($\Prob\I(E)=1$)
but which will falsify the probability measure $\Prob\II$
($\Prob\II(E)=0$).
Part \ref{it:ac-s-2} of Corollary \ref{cor:ac-s}
says that Forecaster II is automatically unreliable
if and only if the Forecasters' predictions are far apart in the sense of (\ref{eq:far})
holding almost surely w.r.\ to the reliable probability measure $\Prob\I$.

The game-theoretic Theorem \ref{thm:criterion}
has several advantages over Corollary \ref{cor:ac-s}.
First of all,
Theorem \ref{thm:criterion} is ``pointwise'':
it carries information about specific plays of the game.
It also has the flexibility provided by the game-theoretic framework in general:
\begin{itemize}
\item
  the Forecasters can react to the Sceptics' moves
  (and Reality can react to both the Forecasters and the Sceptics);
\item
  all players can react to various events outside the protocol;
\item
  game-theoretic results about merging of opinions,
  unlike the standard measure-theoretic results about absolute continuity and singularity,
  do not depend on the (vast) parts of the space $\Omega^*$
  that are never reached or even approached by Reality.
\end{itemize}

\ifnotJOURNAL
\subsection*{Acknowledgments}
\fi
\ifJOURNAL
  \begin{acknowledgements}
\fi
I am grateful to Akio Fujiwara and Phil Dawid for encouragement and useful discussions.
\ifJOURNAL
  \end{acknowledgements}
\fi

\ifarXiv

\fi

\end{document}